\documentclass[11pt, reqno]{amsart}

\author[P.~Leonetti]{Paolo Leonetti}
\address{Universit\`a ``Luigi Bocconi''\\Department of Statistics\\Milan, Italy}
\email{leonetti.paolo@gmail.com}
\urladdr{\url{http://orcid.org/0000-0001-7819-5301}}

\author[C.~Sanna]{Carlo Sanna}
\address{Universit\`a degli Studi di Torino\\Department of Mathematics\\Turin, Italy}
\email{carlo.sanna.dev@gmail.com}
\urladdr{\url{http://orcid.org/0000-0002-2111-7596}}


\title{A note on primes in certain residue classes}

\usepackage{amsmath}
\usepackage{amssymb}
\usepackage{amsthm}
\usepackage{stmaryrd}
\usepackage[left=3.5cm, right=3.5cm, paperheight=11.8in]{geometry}
\usepackage{hyperref}
\usepackage{fancyhdr}
\usepackage{enumitem}
\usepackage{graphicx}
\usepackage[utf8]{inputenc}

\newtheorem{thm}{Theorem}
\newtheorem{lem}[thm]{Lemma}

\theoremstyle{definition}

\hypersetup{
    pdftitle={A note on primes in certain residue classes},
    pdfauthor={Paolo Leonetti and Carlo Sanna},
    pdfmenubar=false,
    pdffitwindow=true,
    pdfstartview=FitH,
    colorlinks=true,
    linkcolor=blue,
    citecolor=green,
    urlcolor=cyan
}

\def\lcm{\operatorname{lcm}}

\AtBeginDocument{%
   \def\MR#1{}
}
                                 
\begin{document}

\maketitle
\thispagestyle{empty}

\begin{abstract} 
Given positive integers $a_1,\ldots,a_k$, we prove that the set of primes $p$ such that $p \not\equiv 1 \bmod{a_i}$ for $i=1,\ldots,k$ admits asymptotic density relative to the set of all primes which is at least $\prod_{i=1}^k \left(1-\frac{1}{\varphi(a_i)}\right)$, where $\varphi$ is the Euler's totient function.
This result is similar to the one of Heilbronn and Rohrbach, which says that the set of positive integer $n$ such that $n \not\equiv 0 \bmod a_i$ for $i=1,\ldots,k$ admits asymptotic density which is at least $\prod_{i=1}^k \left(1-\frac{1}{a_i}\right)$.
\end{abstract}

\section{Introduction}

The \emph{natural density} of a set of positive integers $\mathcal{A}$ is defined as
\begin{equation*}
\mathbf{d}(\mathcal{A}) := \lim_{x \to +\infty} \frac{\#(\mathcal{A} \cap [1, x])}{x} ,
\end{equation*}
whenever this limit exists.
The study of natural densities of sets of positive integers satisfying some arithmetic constraints is a classical research topic.
In particular, Heilbronn~\cite{heilbronn} and Rohrbach~\cite{MR1581555} proved, independently, the following result:

\begin{thm}\label{thm:original}
Let $a_1,\ldots,a_k$ be some positive integers. 
Then, the set $\mathcal{A}$ of positive integers $n$ such that $n\not\equiv 0\bmod{a_i}$ for $i=1,\ldots,k$ has natural density satisfying
\begin{equation*}
\mathbf{d}(\mathcal{A}) \geq \prod_{i=1}^k \left(1-\frac{1}{a_i}\right) .
\end{equation*}
\end{thm}

Generalizations of Theorem \ref{thm:original} were given, for instance, by Behrend~\cite{MR0026081} and Chung~\cite{MR0005511}. 
We refer to \cite{MR1414678} for a textbook expositions and to \cite{MR1366568, MR1438647, MR0026088, MR0422124} for related results.
It is worth noting that Besicovitch~\cite{MR1512943} proved that, given a sequence of positive integers $(a_i)_{i\ge 1}$, the set $\mathcal{A}$ of positive integers $n$ not divisible by any $a_i$ does not necessarily admit natural density. 
However, Davenport and Erd{\H o}s~\cite{MR0043835} proved that $\mathcal{A}$ always admits logarithmic density, i.e., the following limit exists:
\begin{equation*}
\lim_{x \to +\infty} \frac{1}{\log x} \sum_{n \,\in\, \mathcal{A} \,\cap\, [1, x]} \frac{1}{n} .
\end{equation*}

The purpose of this note is to prove a result for the set of primes analogous to Theorem~\ref{thm:original}.
Of course, to this aim, the natural density is not the right quantity to consider, since it is well known that the set of primes has natural density equal to zero.

Define the \emph{relative density} of a set of primes $\mathcal{P}$ as
\begin{equation*}
\mathbf{r}(\mathcal{P}) := \lim_{x \to +\infty} \frac{\#(\mathcal{P} \cap [1,x])}{x / \log x} ,
\end{equation*}
whenever this limit exists.
Furthermore, let $\varphi$ denote the Euler's totient function.

Our result is the following:

\begin{thm}\label{thm:main}
Let $a_1,\ldots,a_k$ be some positive integers. 
Then, the set $\mathcal{P}$ of primes $p$ such that $p\not\equiv 1\bmod{a_i}$ for $i=1,\ldots,k$ has relative density satisfying
\begin{equation*}
\mathbf{r}(\mathcal{P}) \geq \prod_{i=1}^k \left(1-\frac{1}{\varphi(a_i)}\right) .
\end{equation*}
\end{thm}

\section{Preliminaries}

We begin by fixing some notations with the aim of simplifying the exposition.
Let $\mathbf{N}$ be the set of positive integers.
Put $\llbracket x, y \rrbracket := [x, y] \cap \mathbf{N}$ for all $x \leq y$, and let the other ``integral interval'' notations, like $\rrbracket x, y \rrbracket$, be defined in the obvious way. 
For vectors $\mathbf{x} = (x_1, \ldots, x_d)$ and $\mathbf{y} = (y_1, \ldots, y_d)$ belonging to $\mathbf{N}^d$, define $\|\mathbf{x}\| := x_1 \cdots x_d$ and $\llbracket \mathbf{x}, \mathbf{y} \rrbracket  := \llbracket x_1, y_1 \rrbracket \times \cdots \times \llbracket x_d, y_d \rrbracket$.
Also, all the elementary operations of addition, subtraction, multiplication, and division between vectors are meant component-wise, e.g., $\mathbf{x}\mathbf{y} := (x_1 y_1, \ldots, x_d y_d)$.
Let $\mathbf{0}$, respectively $\mathbf{1}$, be the vector of $\mathbf{N}^d$ with all components equal to $0$, respectively $1$, where $d$ will be always clear from the context.
Finally, write $\mathbf{x} \equiv \mathbf{y} \bmod \mathbf{m}$ if and only if $x_i \equiv y_i \bmod m_i$ for all $i = 1, \ldots, d$, where $\mathbf{m} = (m_1, \ldots, m_d) \in \mathbf{N}^d$, and write $\mathbf{x} \not\equiv \mathbf{y} \bmod \mathbf{m}$ if and only if $x_i \not\equiv y_i \bmod m$ for at least one $i \in \llbracket 1, d \rrbracket$.

We will need the following lemma, which might be interesting per se.

\begin{lem}\label{lem:multidimensional}
Let $d$ be a positive integers and let $\mathbf{a}_1, \ldots, \mathbf{a}_k, \mathbf{b} \in \mathbf{N}^d$ be some vectors such that $\mathbf{b} \equiv \mathbf{0} \bmod \mathbf{a}_i$ for $i=1,\ldots,k$. 
Then, the set $\mathcal{X}$ of all $\mathbf{x} \in \llbracket \mathbf{1}, \mathbf{b} \rrbracket$ such that $\mathbf{x} \not\equiv \mathbf{0} \bmod \mathbf{a}_i$ for $i=1,\ldots,k$ satisfies
\begin{equation*}
\#\mathcal{X} \geq \|\mathbf{b}\| \cdot \prod_{i=1}^k \left(1 - \frac1{\|\mathbf{a}_i\|}\right) .
\end{equation*}
\end{lem}
\begin{proof}
Define $\mathbf{c} := \mathbf{a}_1 \cdots \mathbf{a}_k$ and let $\mathcal{Y}$ be the set of $\mathbf{y} \in \llbracket \mathbf{1}, \mathbf{c} \rrbracket$ such that $\mathbf{y} \not\equiv \mathbf{0} \bmod \mathbf{a}_i$ for $i=1,\ldots,k$.
Then, a result of Chung~\cite{MR0005511} says that
\begin{equation}\label{equ:Ybound}
\#\mathcal{Y} \geq \|\mathbf{c}\| \cdot \prod_{i=1}^k \left(1 - \frac1{\|\mathbf{a}_i\|}\right) .
\end{equation}
Clearly, $\mathcal{Y}$ can be partitioned in $\|\mathbf{c} / \mathbf{b}\|$ sets given by
\begin{equation*}
\mathcal{Y}_\mathbf{t} := \left\rrbracket \mathbf{b}(\mathbf{t} - \mathbf{1}), \mathbf{b} \mathbf{t} \right\rrbracket \cap \mathcal{Y},
\end{equation*}
for $\mathbf{t} \in \llbracket \mathbf{1}, \mathbf{c} / \mathbf{b} \rrbracket$.
Therefore, by (\ref{equ:Ybound}) there exists some $\mathbf{t} \in \llbracket \mathbf{1}, \mathbf{c} / \mathbf{b} \rrbracket$ such that
\begin{equation*}
\#\mathcal{Y}_\mathbf{t} \geq \frac{\#\mathcal{Y}}{\|\mathbf{c} / \mathbf{b}\|} \geq \|\mathbf{b}\| \cdot \prod_{i=1}^k \left(1 - \frac1{\|\mathbf{a}_i\|}\right) .
\end{equation*}
Moreover, for each $\mathbf{y} \in \mathcal{Y}_\mathbf{t}$ there exists a unique $\mathbf{x} \in \llbracket \mathbf{1}, \mathbf{b} \rrbracket$ such that $\mathbf{x} \equiv \mathbf{y} \bmod \mathbf{b}$.
Finally, since $\mathbf{b} \equiv \mathbf{0} \bmod \mathbf{a}_i$ for $i = 1, \ldots, k$, it follows easily that the map $\mathbf{y} \mapsto \mathbf{x}$ is an injection $\mathcal{Y}_\mathbf{t} \to \mathcal{X}$, so that $\#\mathcal{X} \geq \#\mathcal{Y}$ and the proof is complete.
\end{proof}

We will also use the following version of Dirichlet's theorem on primes in arithmetic progressions \cite[pag.~82]{Val97}.

\begin{thm}\label{thm:dirichlet}
For all coprime positive integers $a$ and $b$, the set of primes $p$ such that $p \equiv a \bmod b$ has relative density equal to $1 / \varphi(b)$.
\end{thm}

\section{Proof of Theorem~\ref{thm:main}}

Put $\ell := \lcm(a_1, \ldots, a_k)$ and let $\ell = p_1^{e_1} \cdots p_d^{e_d}$ be the canonical prime factorization of $\ell$, where $p_1 < \cdots < p_d$ are primes and $e_1, \ldots, e_d$ are positive integers.
Furthermore, let $\mathcal{S}$ be the set of all $n \in \llbracket 1, \ell \rrbracket$ such that: $n$ is relatively prime to $\ell$, and $n \not\equiv 1 \bmod a_i$ for $i=1,\ldots,k$.
Thanks to Theorem~\ref{thm:dirichlet}, we have
\begin{equation}\label{equ:rP1}
\mathbf{r}(\mathcal{P}) = \lim_{x \to +\infty} \frac{\#(\mathcal{P} \cap [1, x])}{x / \log x} = \lim_{x \to +\infty} \sum_{s \in \mathcal{S}} \frac{\#\{p \leq x : p \equiv s \bmod \ell\}}{x / \log x} = \frac{\#\mathcal{S}}{\varphi(\ell)} ,
\end{equation}
hence the relative density of $\mathcal{P}$ exists, and all we need is the right lower bound for $\#\mathcal{S}$.

For the sake of clarity, let us first assume that $8 \nmid \ell$.
Later, we will explain how to adapt the proof for the case $8 \mid \ell$.
Let $g_i$ be a primitive root modulo $p_i^{e_i}$, for $i=1,\ldots,d$.
Note that $g_1$ exists when $p_1 = 2$ since $e_1 \leq 2$.
Put also $\mathbf{b} := (\varphi(p_1^{e_1}), \ldots, \varphi(p_d^{e_d}))$.
By the Chinese Remainder Theorem, each $n \in \llbracket 1, \ell \rrbracket$ which is relatively prime to $\ell$ is uniquely identified by a vector $\mathbf{x}(n) = (x_1(n), \ldots, x_d(n)) \in \llbracket \mathbf{1}, \mathbf{b} \rrbracket$ such that $n \equiv g_i^{x_i(n)} \bmod p_i^{e_i}$ for $i=1,\ldots,d$.
Let $a_i = p_1^{\alpha_{i,1}} \cdots p_d^{\alpha_{i,d}}$ be the prime factorization of $a_i$, where $\alpha_{i,1}, \ldots,\alpha_{i,d}$ are nonnegative integers, and define $\mathbf{a}_i := (\varphi(p_1^{\alpha_{i,1}}), \ldots, \varphi(p_d^{\alpha_{i,d}}))$ for $i=1,\ldots,k$.

At this point, it follows easily that $n \in \mathcal{S}$ if and only if $\mathbf{x}(n) \in \mathcal{X}$, where $\mathcal{X}$ is the set in the statement of Lemma~\ref{lem:multidimensional}.
Hence, the map $n \mapsto \mathbf{x}(n)$ is a bijection $\mathcal{S} \to \mathcal{X}$ and, as a consequence, $\#\mathcal{S} = \#\mathcal{X}$.
Since $\|\mathbf{b}\| = \varphi(\ell)$, $\|\mathbf{a}_i\| = \varphi(a_i)$, and $\mathbf{b} \equiv \mathbf{0} \bmod \mathbf{a}_i$ for $i=1,\ldots,k$, the desired claim follows from Lemma~\ref{lem:multidimensional} and (\ref{equ:rP1}).

The case $8 \mid \ell$ is a bit more trickier since there are no primitive roots modulo $2^e$, for $e \geq 3$ an integer.
However, the previous proof still works by putting 
\begin{equation*}
\mathbf{b} := (2, 2^{e_1-2}, \varphi(p_2^{e_2}), \ldots, \varphi(p_d^{e_d}))
\end{equation*}
and
\begin{equation*}
\mathbf{a}_i := (2^{\max(0,\alpha_{i,1} - 1) - \max(0, \alpha_{i,1} - 2)}, 2^{\max(0, \alpha_{i,1} - 2)}, \varphi(p^{\alpha_{i,2}}), \ldots, \varphi(p^{\alpha_{i,d}}))
\end{equation*}
for $i=1,\ldots,k$.
Now each $n \in \llbracket 1, \ell \rrbracket$ which is relatively prime to $\ell$ is uniquely identified by a vector $\mathbf{x}(n) = (x_0(n), \ldots, x_d(n)) \in \llbracket \mathbf{1}, \mathbf{b} \rrbracket$ such that $n \equiv (-1)^{x_0(n)} 5^{x_1(n)} \bmod 2^{e_1}$ and $n \equiv g_i^{x_i(n)} \bmod p_i^{e_i}$ for $i=2,\ldots,d$.
The rest of the proof proceeds as before.

\bibliographystyle{amsplain}

\end{document}